\documentclass[a4paper, 11pt, reqno]{amsart}
\usepackage{enumerate}

\usepackage{hyperref}
\hypersetup{colorlinks=true,linkcolor=blue,citecolor=blue,pdfpagemode=UseNone,pdfstartview={XYZ null null 1.00}}

\usepackage{amsmath, amsthm, amsopn, amssymb}
\usepackage{xcolor}
\usepackage{graphicx}

\usepackage{geometry}
\theoremstyle{plain}
\newtheorem*{theorem*}{Theorem}
\newtheorem{theorem}{Theorem}[section]
\newtheorem{lemma}[theorem]{Lemma}

\newtheorem*{claim*}{Claim}
\newtheorem{corollary}[theorem]{Corollary}
\newtheorem{conjecture}[theorem]{Conjecture}

\newtheorem{question}[theorem]{Question}
\theoremstyle{remark}

\newcommand{\F}{\mathcal{F}}
\newcommand{\HH}{\mathcal{H}}
\newcommand{\PP}{\mathcal{P}}

\newcommand{\defi}{\text{def}}

\let\emptyset\varnothing
\let\eps\varepsilon
\let\originalleft\left
\let\originalright\right
\renewcommand{\left}{\mathopen{}\mathclose\bgroup\originalleft}
\renewcommand{\right}{\aftergroup\egroup\originalright}

\title{Completion and deficiency problems}

\author{Rajko Nenadov}
\thanks{Department of Mathematics, ETH, Z\"urich, Switzerland. Email: \href{mailto:rajko.nenadov@math.ethz.ch} {\nolinkurl{rajko.nenadov@math.ethz.ch}}.
\\ Research supported in part by SNSF grant 200021-175573.}

\author{Benny Sudakov}
\thanks{Department of Mathematics, ETH, Z\"urich, Switzerland. Email:
\href{mailto:benjamin.sudakov@math.ethz.ch} {\nolinkurl{benjamin.sudakov@math.ethz.ch}}.
\\ Research supported in part by SNSF grant 200021-175573.}

\author{Adam Zsolt Wagner}
\thanks{Department of Mathematics, ETH, Z\"urich, Switzerland. Email:
\href{mailto:zsolt.wagner@math.ethz.ch} {\nolinkurl{zsolt.wagner@math.ethz.ch}}.}

\begin{document}
\begin{abstract}
    Given a partial Steiner triple system (STS) of order $n$, what is the order of the smallest complete STS it can be embedded into? The study of this question goes back more than 40 years. In this paper we answer it for relatively sparse STSs, showing that given a partial STS of order $n$ with at most $r \le \eps n^2$ triples, it can always be embedded into a complete STS of order $n+O(\sqrt{r})$, which is asymptotically optimal. We also obtain similar results for completions of Latin squares and other designs.

    This suggests a new, natural class of questions, called \emph{deficiency problems}. Given a global spanning property $\PP$ and a graph $G$, we define the deficiency $\defi(G)$ of the graph $G$ with respect to the property $\PP$ to be the smallest positive integer $t$ such that the join $G\ast K_t$ has property $\PP$. To illustrate this concept we consider  deficiency versions of some well-studied properties, such as having a $K_k$-decomposition, Hamiltonicity,  having a triangle-factor and having a perfect matching in hypergraphs. 
    
    The main goal of this paper is to propose a systematic study of these problems; thus several future research directions are also given.
\end{abstract}

\maketitle

\section{Completion problems}

\subsection{Steiner triple systems and $(n,k)$-block designs} \label{subsec:design}

\emph{Steiner triple systems} (STSs) are one of the most classical objects studied in combinatorial design theory, dating back to Kirkman~\cite{kirkman}. We say that a family $\F$ of $3$-element subsets, called \emph{blocks}, of an $n$-element set $X$ is a STS if any pair of distinct elements of $X$ are contained in precisely one block. In 1847, Kirkman~\cite{kirkman} proved that a Steiner triple system exists if and only if $n\equiv 1,3 \text{ (mod }6)$. 

A family $\F$ of $3$-element subsets of an $n$-element set $X$ is a \emph{partial STS} if any pair of distinct elements of $X$ is contained in at most one block. Given a partial STS $\F$ on a set $X$, we say that a (complete) STS $\F'$ on a set $X'$ is an \emph{embedding} of $\F$ if $X \subseteq X'$ and $\F \subseteq \F'$. This naturally prompts the following question: can every partial STS $\F$ be obtained from some STS $\F'$ by deleting a number of elements, or, equivalently, does every partial STS have an embedding; if yes, then what is the order\footnote{An order of a (partial) STS is the size of the underlying set $X$.} of a smallest such STS? In 1977 Lindner conjectured~\cite{lindner} that any partial Steiner system of order $n$ has an embedding of order $n'$, for each $n'\geq 2n+1$ such that $n'\equiv 1,3 \text{ (mod }6)$. After some progress over the years~\cite{lindner1,lindner2,lindner3,lindner4} the conjecture was recently proved by Bryant and Horsley~\cite{lindnerproof}. While the bound $2n+1$ is sharp in general, it is a natural question whether this can be improved if $\F$ is sparse, that is if $\F$ contains only a few blocks. Some results along these lines were obtained in~\cite{colbourn1983completing,bryant2002conjecture,bryant2004existence,horsley2014embedding,horsley2014small}. In particular, Horsley~\cite{horsley2014embedding} showed that if $\F$ has at most $n^2/50 - o(n^2)$ blocks then it has an embedding of order $8n/5 + O(1)$. In this paper we determine an asymptotically optimal bound for embeddings of relatively sparse partial STSs.
\begin{theorem}\label{thm:steiner}
There exist absolute constants $\eps, n_0 > 0$ such that the following holds. If $\F$ is a partial Steiner triple system of order $n \ge n_0$ with $|\F|\leq \eps n^2$ blocks, then there exists an embedding of $\F$ of order at most $n + O(\sqrt{|\F|})$.
\end{theorem}

Similarly to Steiner triple systems, an $(n,k)$-design is a collection of blocks of size $k$ covering each pair of elements exactly once (and so a STS is a $(n,3)$-design). Here the notions of a partial $(n,k)$-design and an embedding are defined analogously as in the case of Steiner triple systems (i.e.\ $k = 3$). Wilson~\cite{wilson} proved that the necessary divisibility conditions suffice for the existence of such designs. Embedding problems of partial $(n,k)$-designs, for $k\geq 4$, have also been studied, see e.g.~\cite{ganterquad,phelpssurvey}. We obtain the following generalization of Theorem~\ref{thm:steiner}:

\begin{theorem}\label{thm:kkdecomp}
For every integer $k \geq 3$, there exist $\eps, n_0 > 0$ such that the following holds. If $\F$ is a partial $(n,k)$-design of order $n \ge n_0$ with $|\F| \le \eps n^2$ blocks, then there exists an embedding of $\F$ of order at most $n + 7k^2\sqrt{\F}$.
\end{theorem}

\noindent
Theorem~\ref{thm:kkdecomp} (and therefore Theorem~\ref{thm:steiner}) is tight up to a multiplicative constant in the number of added elements with respect to both $k$ and $|\F|$ for every $k \ge 3$ and $|\F| \ge 2n$. We present a construction which shows this in Section~\ref{sec:designs}. For $|\F| < n/2$ and $k = 3$, Horsley~\cite{horsley2014embedding} conjectured that in fact there exists an embedding of $\F$ which is of the same order as $\F$.

Note that a STS on a set of size $n$ corresponds to a $K_3$-decomposition of a complete graph on $n$ vertices. Recall that a $K_k$-decomposition of a graph $G$ is a partition of the edge set of $G$ into copies of $K_k$. 

\subsection{Latin squares}

Completion problems can also be studied for Latin squares, which are another classical combinatorial objects dating back to Ozanam~\cite{ozanam} and Euler~\cite{euler}. Recall that an $n \times n$ matrix $M$ with entries in $[n]$ forms a Latin square if $M_{i,j} \neq M_{i', j}$ and $M_{i,j} \neq M_{i,j'}$ for all $i\neq i'$ and $j\neq j'$. In other words, every element of $[n]$ appears exactly once in each row and each column. To be consistent with terminology from Section \ref{subsec:design}, we say that $M$ is of order $n$. Notions of partial Latin squares and embeddings are also defined analogously, as follows. An $n \times n$ matrix $P$ over $[n] \cup \{\ast\}$ forms a \emph{partial} Latin square if $P_{i', j} \neq P_{i,j} \neq P_{i, j'}$ for all $i\neq i' $ and  $j\neq j'$ such that $P_{i,j} \neq \ast$. Thinking of $\ast$ as the symbol for `empty', we have that every element of $[n]$ appears at most once in every row and column. Let us denote with $|P|$ the number of non-empty cells in $P$. Given a partial Latin square $P$ over $[n]$ and a (complete) Latin square $M$ over $[n']$, for some $n' \ge n$, we say that $M$ is an \emph{embedding} of $P$ if $M_{i,j} = P_{i,j}$ for every $1\leq i,  j\leq n$ such that $P_{i,j} \neq \ast$. If $M$ and $P$ are of the same order, that is $n' = n$, we say that $M$ is a \emph{completion} of $P$.


Completions and embeddings of Latin squares have been extensively studied as well and have a long history, see e.g.~\cite{ryser, cruse, latin3, latin4}. A classical theorem of Smetaniuk~\cite{smetaniuk} as well as Anderson and
Hilton~\cite{andhilt} states that every partial Latin square of order $n$ with at most $n-1$ non-empty cells has a completion. This is the best possible result in the sense that there exists a partial Latin square of order $n$ with exactly $n$ non-empty cells which does not have a completion: put $1$ on every cell of the main diagonal except the last one, and in the last one put $2$. Thus in order to be able to complete a partial Latin square with more than $n-1$ empty cells we need to impose further structural restrictions. One such possible restriction was conjectured by Daykin and H\"aggkvist~\cite{dayhag}: every partial Latin square such that each row, column and symbol are used at most $n/4$ times (that is, each row and column contain at least $3n/4$ elements $\ast$) has a completion. For discussion and sharpness examples see~\cite{wanless}. 

Evans~\cite{evans} showed that  any partial Latin square  of order $n$ has an embedding of order $2n$, and this bound is sharp. Our next theorem solves asymptotically the completion problem for sparse Latin squares.
\begin{theorem}\label{thm:latin}
There exist absolute constants $C,\eps,n_0 >0$ such that the following holds. If $L$ is a partial Latin square of order $n\geq n_0$ with $|L| \le \eps n^2$, then $L$ has an embedding of order $n'$ for some $n'\leq n + C\sqrt{|L|}$.
\end{theorem}

\noindent
This theorem is tight up to the value of the constant $C$ (see the discussion following the proof). Notice that a (complete) Latin square corresponds to a $K_3$-decomposition of the complete tripartite graph $K_{n,n,n}$. Thus Theorem~\ref{thm:latin} can be seen as a \emph{multi-partite} analogue of Theorem \ref{thm:steiner}. Next, we discuss a Latin square analogue of Theorem \ref{thm:kkdecomp}.

We say that two (partial) Latin squares $R$ (red) and $B$ (blue) of order $n$ are \emph{orthogonal} if no two cells contain the same combination of red symbol and blue symbol, unless one of these cells is empty. More precisely, we have $(R_{i,j}, B_{i,j}) \neq (R_{i',j'}, B_{i', j'})$ for every $1\leq i, i', j, j' \leq n$ such that $(i,j) \neq (i', j')$ and $R_{i,j}, R_{i', j'}, B_{i,j}, B_{i',j'} \neq \ast$. Given a family $\PP = \{P_1, \ldots, P_r\}$ of partial Latin squares of order $n$, let us denote with $|\PP|$ the number of cells $(i,j) \in [n]^2$ such that $(P_k)_{i,j} \neq \ast$ for some $k \in [r]$. In other words, if we `draw' all Latin squares from $\PP$ into the same $n \times n$ grid, then $|\PP|$ denotes the number of cells which contain at least one non-empty symbol.

 In~\cite{barber} it was proved that if we are given $r$ mutually orthogonal partial Latin  squares $P_1, \ldots, P_r$ of order $n$ such that in each row and column all but at most $c_r n$ cells are empty in every $P_i$ (where $c_r$ is a specific small constant depending on $r$), then they can be completed to a set of mutually orthogonal Latin squares. We show that if we drop the condition that every row and column has a bounded number of non-empty cells, but instead ask only for not too many cells to be filled in, there exist embeddings $P_1', \ldots, P_r'$ of $P_1, \ldots, P_r$ respectively, each of order $n'$ for some $n'$ which is only slightly larger than $n$, which are also pairwise orthogonal.

\begin{theorem}\label{thm:latinort}
For every integer $r\geq 3$ there exists positive $\eps$ and integer $ n_0$ such that the following holds. If $\PP = \{P_1, \ldots, P_{r-2}\}$ are $r-2$ pairwise orthogonal partial Latin squares of order $n$ such that $ |\PP| \le \eps n^2$, then there exist pairwise orthogonal Latin square $P_1', \ldots, P_{r-2}'$ of order $n' \le n + Cr^2\sqrt{|\PP|}$ such that each $P_i'$ is an embedding of $P_i$. 
\end{theorem}

\noindent
As with the previous theorems, Theorem~\ref{thm:latinort} is tight up to the value of $C$ and a sharpness example is presented after the proof. Note that for $r = 3$ we have that Theorem~\ref{thm:latinort} is equivalent to Theorem~\ref{thm:latin}.

To see the connection between Theorem~\ref{thm:latinort} and Theorem~\ref{thm:kkdecomp}, observe that a pair of orthogonal complete Latin squares corresponds to a $K_4$-decomposition of $K_{n,n,n,n}$, a $4$-partite complete graph with each part being of size $n$. In general, a set of $r-2$ pairwise orthogonal complete $n\times n$ Latin squares corresponds to a $K_r$-decomposition of the complete $r$-partite graph with vertex classes of size $n$ (this connection is discussed in detail in Section~\ref{sec:latin}). Thus, not surprisingly, the proof of Theorem~\ref{thm:latinort} will follow the idea of the proof of Theorem~\ref{thm:kkdecomp}. However, there are some additional difficulties in this case which did not occur in Theorem~\ref{thm:kkdecomp}.

\section{A general class of completion problems and the notion of deficiency}
Our results in Section 1 suggest the following new class of extremal problems. Given a global, spanning property $\PP$ (e.g.~Hamiltonicity) it is easy to see that by deleting even some small number of edges from the complete graph we can obtain a graph which does not satisfy $\PP$. The reason is that isolating a vertex can be done cheaply by deleting only $n-1$ edges. This is why historically Tur\'an type problems were mostly studied for local properties (e.g.~containing a triangle) whereas for global properties a minimum degree condition was usually added to avoid the above issues. In the present paper we look at this problem differently.  

For a graph $G$ and integer $t\geq 0$, denote by $G\ast K_t$ the \emph{join} of $G$ and $K_t$, which is the graph obtained from $G$ by adding to its vertex set $t$ new vertices and adding every edge that is incident to at least one of the $t$ new vertices. That is, $G\ast K_t$ has $|V(G)|+t$ vertices and $e(G)+\binom{t}{2}+t|V(G)|$ edges. Similarly if $\mathcal{H}$ is a $k$-uniform hypergraph, denote by $\mathcal{H}\ast K_t$ the $k$-uniform hypergraph obtained by adding $t$ new vertices and all $k$-edges containing at least one of the $t$ new vertices.

Observe that Theorem~\ref{thm:kkdecomp} can be equivalently formulated as follows: if one removes up to $r \le \eps n^2$ edge-disjoint copies of $K_k$ from $K_n$ to obtain a graph $G$, then there exists some $t$ with $t\leq Ck^2\sqrt{r}$ so that $G\ast K_t$ has a $K_k$-decomposition. Moreover, as a set of $r$ orthogonal Latin squares can be viewed as a $K_{r+2}$-decomposition of a complete balanced $r+2$-partite graph, Theorem~\ref{thm:latinort} is the multipartite analogue of Theorem~\ref{thm:kkdecomp}. 

In general, given a property $\PP$ and a graph $G$, we propose to study the minimum positive integer $t$ such that the join $G\ast K_t$ has property $\PP$. We call such $t$ the \emph{deficiency} of the graph $G$ with respect to the property $\PP$. Our previous theorems address the deficiency of graphs with respect to the existence of a $K_k$-decomposition and an extension of this problem to the multi-partite setting. Note that the concept of deficiency appeared before, for example in the Tutte--Berge formula on the characterization of the size of a maximum matching in a graph. Nevertheless the questions we propose here are new and the study of deficiency concept by itself leads to intriguing open problems. We illustrate this with two more examples: the existence of a Hamiltonian cycle and the existence of a triangle-factor.

\subsection{Hamiltonian graphs}

It is clear that an $n$-vertex graph may have as many as $\binom{n}{2}-(n-2)$ edges without having a Hamiltonian cycle.  E.g.,~take a complete graph and remove all but one edge incident to some chosen vertex. This was proved to be tight by Ore~\cite{ore} in 1961. The deficiency variant of this problem exhibits a more interesting behaviour. Given integers $n$ and $m$, what is the smallest integer $f(n,m)$ such that $G\ast K_{f(n,m)}$ has a Hamiltonian cycle for any $n$-vertex graph $G$ with $m$ edges? Equivalently, given that $G\ast K_t$ does not have a Hamiltonian cycle, how many edges can $G$ have?

Here we give a complete answer to this question and also determine the extremal constructions. It appears that there are two natural and competing constructions, both of which are best for some range of $t$.
\begin{theorem}\label{thm:hami}
Let $n$ and $t$ be integers and $G$ an $n$-vertex graph so that $G\ast K_t$ does not have a Hamiltonian cycle. Then we have the following bounds on $e(G)$.
\begin{itemize}
    \item If $n+t$ is even:
    $$
        e(G) \le \binom{n}{2} - \begin{cases}
            t(n-1) - \binom{t}{2} & \text{ if } t \le (n+4)/5 \\
            \binom{\frac{n+t+2}{2}}{2} - 1 & \text{ if } t \ge (n+4)/5.
        \end{cases}
    $$
    
    \item If $n + t$ is odd:
    $$
        e(G) \le \binom{n}{2} - \begin{cases}
            t(n-1) - \binom{t}{2} & \text{ if } t \le (n+1)/5 \\
            \binom{\frac{n+t+1}{2}}{2}  &\text{ if } t \ge (n+1)/5.
        \end{cases}
    $$
\end{itemize}
In the case where $n + t$ is even and $t = (n+4)/5$, or $n+t$ is odd and $t = (n+1)/5$, there are exactly two graphs which achieve equality. In all other cases there is a unique $G$ for which the equality holds.
\end{theorem}

\noindent
These upper bounds on the number of edges of $G$ such that $G\ast K_t$ does not contain a Hamilton cycle were implicitly obtained by Skupie\'n~\cite{skupien}.

We remark that the deficiency problem for Hamiltonicity is equivalent to another natural question. Define the path-covering number $\mu(G)$ of a graph $G$ to be the minimum number of vertex-disjoint paths required to cover the vertices of $G$. Let $g(n,k)$ be the minimum integer so that every $n$-vertex graph $G$ with at least $g(n,k)$ edges has $\mu(G)\leq k$. This parameter $g(n,k)$ has been studied before, see e.g.~\cite{pathcover1, pathcover2}. Note that a non-Hamiltonian graph has path-cover number $t$ precisely if $G\ast K_t$ has a Hamiltonian cycle but $G\ast K_{t-1}$ does not. Therefore Theorem~\ref{thm:hami} gives a tight bound on $g(n,k)$ for all values of the parameters together with extremal examples.

\subsection{Triangle factors}\label{sec:trianfac}

Given an integer $k\geq 3$, a $K_k$-factor of a graph $G$ is a collection of vertex-disjoint copies of $K_k$ covering every vertex of $G$. Corr\'adi and Hajnal~\cite{corradi} proved that if a graph on $3n$ vertices has minimum degree $\delta(G)\geq 2n$ then it contains a triangle factor, and a corresponding minimum degree condition for the existence of a $K_k$-factor was determined later by Hajnal and Szemer\'edi~\cite{hajszem}. The deficiency problem for $K_k$-factors asks: given integers $n$ and $m$, what is the smallest integer $f(n,m)$ such that $G\ast K_{f(n,m)}$ has a $K_k$-factor for any $n$-vertex graph $G$ with $m$ edges? Equivalently, given that $G\ast K_t$ does not have a $K_k$-factor and $k|n+t$, how many edges can $G$ have? We give a partial answer to this problem in the case $k=3$. 
\begin{theorem}\label{thm:corrhajdef}
There exists $n_0$ such that the following holds. Let $n \ge n_0$ and $t$ be integers so that $3|n+t$ and let $G$ be an $n$-vertex graph so that $G\ast K_t$ does not have a $K_3$-factor. If $t\leq n/1000$   then 
$$
e(G)\leq \binom{n}{2} - \binom{k}{2} - 
\begin{cases}
    k(n-k), &\text{if } t \text{ is odd,} \\
    k(n-k-1), &\text{if } t \text{ is even,}
\end{cases}
$$
where $k = \lceil (t+1)/2 \rceil$. This bound on $e(G)$ is sharp.
\end{theorem}

\noindent
The range of values of $t$ for which this theorem holds can be easily extended using our proof, but we did not optimize it for the sake of clarity of presentation.

\subsection{Organization of the paper} The rest of this paper is organized as follows. We give the proof of Theorem~\ref{thm:kkdecomp} in Section~\ref{sec:designs}. 
Our main result on Latin squares, Theorem~\ref{thm:latinort}, is proved in Section~\ref{sec:latin}. We discuss the deficiency of the Hamiltonicity property and prove Theorem~\ref{thm:hami} in Section~\ref{sec:hami}, and continue on with the deficiency of $K_3$-factors and the proof of Theorem~\ref{thm:corrhajdef} in Section~\ref{sec:corrhaj}.  We finish with some concluding remarks and suggest future directions of research in Section~\ref{sec:conc}. Whenever rounding is not crucial, we omit it for the sake of brevity.


\section{Completing Steiner triple systems and other block designs}\label{sec:designs}

In this section we prove Theorem~\ref{thm:kkdecomp}. We start with a lemma which may be of independent interest. Hajnal and Szemer\'edi~\cite{hajszem} proved that every graph $G$ with  $sk$ vertices and minimum degree $\delta(G) \geq s(k - 1)$ contains a $K_k$-factor, i.e.~a collection of vertex-disjoint $k$-cliques covering every vertex, and this is optimal. In Section \ref{sec:corrhaj} we discuss the deficiency version of this result: if instead of the minimum degree of $G$ we only know its number of edges, how many full degree vertices one needs to add to guarantee the existence of a $K_k$-factor? Our first lemma is similar in spirit to the above problem, however instead of adding full degree vertices we add vertices of very large degree. 

\begin{lemma}\label{lem:defHajSze}
Let $k \ge 3$ and $r\geq 1$ be integers and let $G$ be a graph on vertex set $S~\dot\cup~ T$ with $|T|\geq 5k^2 \sqrt{r}$, such that the degree of every vertex satisfies $d(v)\geq |V(G)|-k\sqrt{r}$. Then no matter how one removes at most $k^2 r$ edges from $G[S]$, the resulting graph contains a collection of vertex-disjoint $k$-cliques covering every vertex of $G$, except possibly up to $k-1$ vertices that lie in $T$.
\end{lemma}

Observe that the bound on $T$ cannot be significantly improved: let $|S| = k \sqrt{r}$ and $|T|<k(k-1) \sqrt{r}$, and suppose $G$ is a complete graph on vertex set $S~\dot\cup~T$. Then removing all edges from $S$ creates a graph without a collection of vertex-disjoint $K_k$'s covering every vertex of $S$.

\begin{proof}[Proof of Lemma~\ref{lem:defHajSze}]
Let $R$ be the set of at most $k^2 r$ edges we deleted from $S$ and set $G_R = G \setminus R$. Let $B \subseteq S$ be the set of vertices incident to at least $2k \sqrt{r}$ edges in $R$. Note that $|B|\leq k\sqrt{r}$. We will first consider the vertices in $B$ one by one, and for each $b\in B$ we will find a disjoint set of $k-1$ vertices in $N_{G_R}(b)\cap T = N_{G}(b) \cap T$ that form a copy of $K_{k-1}$ (recall that $R$ contains only edges which completely lie in $S$).

Let $B=\{b_1,b_2,\ldots,b_{|B|}\}$ and suppose that for each $j\leq i$ we have found a set $B_j\subset T\cap N_G(b_j)$ of $k-1$ vertices so that $G[B_j]\cong K_{k-1}$ and $B_{j_1}\cap B_{j_2}=\emptyset$ for each $1\leq j_1<j_2\leq i$.  Set $T_{i}:=T\setminus\left\{\bigcup_{j\leq i} B_j\right\}$. From $d_G(b_{i+1})\geq |V(G)|-k\sqrt{r}$ we get 
$$
	|N_G(b_{i+1})\cap T_i|\geq |T_i|-k\sqrt{r} \geq |T| - |B|(k-1) - k\sqrt{r} > 3k^2\sqrt{r}.
$$ 
Let $H_i=G[N(b_{i+1})\cap T_i]$ and note that every vertex $v\in H_i$ satisfies 
$$
	d_{H_i}(v)\geq |N(b_{i+1})\cap T_i|- k\sqrt{r} > \left(1-\frac{1}{k-1}\right)|V(H)|.
$$
Therefore, by Tur\'an's theorem $H_i$ contains a copy of $K_{k-1}$. Let $B_{i+1}$ be the vertex set of such a copy. Continuing this process, we have found a collection of vertex-disjoint $K_k$'s covering every vertex of $B$.

It remains to find a collection of disjoint $K_k$'s which cover all the vertices in $S' = S \setminus B$ and all but at most $k - 1$ in $T_{|B|}$. To this end, let $T' \subseteq T_{|B|}$ be an arbitrary subset such that $|S' \cup T'|$ is divisible by $k$ and $|T_{|B|} \setminus T'| \le k - 1$. We will show that the graph $H':=G[S' \cup T']$ contains a $K_k$-factor. Having the Hajnal--Szemer\'edi theorem in mind, stated in the beginning of this section, the only thing we need to show is that $H'$ has sufficiently large minimum degree. Consider some vertex $v \in H'$. As $v \notin B$ we have that it is adjacent to at most $2 k \sqrt{r}$ edges in $R$. Together with the assumption that the initial degree of $v$ (in $G$) was at least $|V(G)| - k \sqrt{r}$, this implies
$$
	d_{H'}(v)\geq |V(H')| - k \sqrt{r} - 2k\sqrt{r} - |T_{|B|} \setminus T'| \geq \left(1-\frac{1}{k}\right)|V(H')|.
$$
In the last inequality we used that 
$$
	|V(H')| \ge |T'| \geq |T|- (|B| + 1)(k-1) > 4k^2\sqrt{r}.
$$
Hence by the Hajnal--Szemer\'edi theorem $H'$ has a $K_k$-factor. Finally, this $K_k$-factor together with the copies of $K_k$ given by $B_i\cup\{b_i\}$ for $1\leq i \leq |B|$, forms a desired collection of vertex-disjoint $K_k$'s.
\end{proof}

To prove Theorem~\ref{thm:kkdecomp} we also need the following result of Gustavsson~\cite{gustavsson}:

\begin{theorem}\label{thm:gust}
For every integer $k\geq 3$ there exists an integer $n_0 > 0$ and a positive constant $\gamma$ such that every graph $G$ with $n\geq n_0$ vertices and minimum degree $\delta(G)\geq (1-\gamma) n$, satisfying $\binom{k}{2}|e(G)$ and $k-1|d(v)$ for every $v \in G$, has a $K_k$-decomposition.
\end{theorem}

\noindent
It is worth mentioning that currently best known bounds on the value of $\gamma$ in Theorem \ref{thm:gust} are due to Dross~\cite{dross2016fractional} ($k = 3$) and Montgomery~\cite{montgomery2017fractional} ($k \ge 4$) together with result from \cite{GKLMO}, which implies that the decomposition threshold for cliques equals its fractional relaxation. As $\gamma$ only has an impact on the value of $\eps$ in Theorem \ref{thm:kkdecomp}, either of these theorems serve our purpose.

The following lemma directly implies Theorem~\ref{thm:kkdecomp}.

\begin{lemma}\label{lem:kkdecomp}
For every integer $k\geq 3$ there exist $\varepsilon, n_0 > 0$ such that the following holds. Let $G$ be a graph obtained from the complete graph on $n \ge n_0$ vertices by deleting $r \le \eps n^2$ edge-disjoint copies of $K_k$. Then there exists some $t \leq 7k^{2}\sqrt{r}$ such that $G\ast K_t$ has a $K_k$-decomposition. 
\end{lemma}
\begin{proof}


Let $j \ge 0$ be the smallest integer such that $n+t$ is divisible by $k$ and $n+t-1$ is divisible by $k-1$, where $t = 6k^2 \sqrt{r} + j$. It follows from the Chinese Remainder Theorem that $j \le k^2$, thus $t \le 7 k^2 \sqrt{r}$. We show that $G' = G \ast K_t$ has a $K_k$-decomposition. For the rest of the proof let $W$ denote the set of $t$ vertices added to $G$ (corresponding to $K_t$) and $R$ the set of edges corresponding to deleted copies of $K_k$.

Let $B$ be the set of vertices in $G$ with degree less than $n - k^2 \sqrt{r}$. As $|R| = r \binom{k}{2}$, we have $|B| \le \sqrt{r}$. By adding arbitrary set of vertices from $G$ to $B$, we can assume that $|B| = \sqrt{r}$. Our first aim is to find a small collection of edge-disjoint copies of $K_k$ in $G'$ that cover every edge incident to at least one vertex of $B$. Removing these $K_k$'s and the vertices in $B$, we will show that the resulting graph has a very high minimum degree and hence by Gustavsson's theorem (Theorem~\ref{thm:gust}) it has a $K_k$-decomposition. Overall, this gives us a $K_k$-decomposition of $G'$.

We build the collection of edge-disjoint copies of $K_k$ in $G'$ that cover every edge incident to at least one vertex of $B$ iteratively, considering the vertices of $B$ one by one. For that, let $B=\{b_1,b_2,\ldots,b_{|B|}\}$ be an arbitrary ordering of the vertices of $B$ and suppose we have defined the collections $\mathcal{S}_1,\mathcal{S}_2,\ldots,\mathcal{S}_i$ of distinct copies of $K_k$ in $G'$, for some $i < |B|$, such that the following holds:
\begin{enumerate}[(i)]
    \item\label{item:disj} for every $j_1, j_2 \in \{1, \ldots, i\}$, if $F_1 \in \mathcal{S}_{j_1}$ and $F_2 \in \mathcal{S}_{j_2}$ with $F_1\neq F_2$ then $F_1$ and $F_2$ are edge-disjoint, and
    \item\label{item:isol} for every $1\leq j\leq i$, every edge in $G'$ incident to $b_j$ belongs to some copy of $K_k$ in $\mathcal{S}_j$ and every copy of $K_k$ in $\mathcal{S}_j$ contains $b_j$.
\end{enumerate}
Set $G_i$ to be the graph obtained from $G'$ by deleting all edges in the $K_k$'s from $\bigcup_{j\leq i}\mathcal{S}_j$. By~(\ref{item:isol}) we have that the vertices $b_1,b_2,\ldots,b_i$ are isolated in $G_i$. By~(\ref{item:disj}) and (\ref{item:isol}) we have that, for every $1\leq j \leq i$, every vertex $v\neq b_j$ occurs in at most one of the $K_k$'s in $\mathcal{S}_j$. Hence by forming $G_i$ from $G'$, the degree of every vertex not in $\{b_1,b_2,\ldots,b_i\}$ reduced by at most 
\begin{equation} \label{eq:degree_red}
	(k - 1) i \le (k - 1)(|B| - 1) \leq (k - 1)(\sqrt{r} - 1) \le (k - 1)\sqrt{r} - 1.
\end{equation}
As the vertex $b_{i+1}$ was connected to every vertex in $W$ (in $G'$), letting $N_i(v)$ denote the neighborhood of a vertex $v$ in $G_i$ we have 
$$
	|N_i(b_{i+1})\cap W|\geq |W|- (k-1)\sqrt{r} + 1 \ge 5 k^2 \sqrt{r}.
$$ 
Set $S_i:=N_i(b_{i+1})\setminus W$ and $T_i:=N_i(b_{i+1})\cap W$. Let $H_i$ be the graph on the vertex set $S_i \cup T_i$ obtained from $G'[S_i \cup T_i]$ by adding back the edges from $R$. By \eqref{eq:degree_red}, every vertex $v \in H_i$ has degree at least 
$$
	\deg_{H_i}(v) \ge |V(H_i)|-1- (k-1) \sqrt{r} + 1 = |V(H_i)|- (k-1)\sqrt{r}.
$$ 
Hence we may remove again edges from $R$ and apply Lemma~\ref{lem:defHajSze}, with $k$ playing the role of $k-1$, to find a collection $\F_i$ of vertex-disjoint $K_{k-1}$'s in $H_i$ covering every vertex of $S_i \cup T_i$ except possibly up to $k-2$ vertices that lie in $T_i$. Note that we can indeed do that as $|R| = r \binom{k}{2} < (k-1)^2r$. Moreover, as for every $w \in V(G') \setminus (S_i \cup T_i)$ we have that the edge $b_{i+1} w$ belongs to a copy of $K_k$ from a collections of edge-disjoint $K_k$'s, from the fact that $k - 1$ divides $|V(G')| - 1$ we conclude that $|S_i \cup T_i|$ is also divisible by $k-1$. Thus $\F_i$ in fact covers every vertex in $S_i \cup T_i$. The copies of $K_{k-1}$ from $\F_i$ together with $b_{i+1}$ form a collection $\mathcal{S}_{i+1}$ of copies of edge-disjoint $K_k$'s covering every edge incident to $b_{i+1}$ in $G_i$, as desired.

Next, let $G^\ast$ be the graph obtained from $G'$ by deleting all edges in the $K_k$'s from $\bigcup_{j\leq |B|}\mathcal{S}_j$, as well as the (now isolated) vertices $b_1,b_2,\ldots,b_{|B|}$. By forming $G^\ast$ from $G'$, the degree of every vertex not in $\{b_1,b_2,\ldots,b_{|B|}\}$ reduced by at most $k \sqrt{r}$ (see \eqref{eq:degree_red}) and hence, by the definition of $B$, the minimum degree of $G^\ast$ is at least 
$$
	\delta(G^\ast) \ge |V(G^\ast)| - k^2 \sqrt{r} - k\sqrt{r} \ge |V(G^\ast)| - 2 k^2 \sqrt{r}.
$$
As $r \le \eps n^2$ and $|V(G^\ast)| = n + t - |B| > n$, for $\eps$ sufficiently small compared to $k$ we have $2k^2 \sqrt{r} < \gamma |V(G^\ast)|$, thus Gustavsson's theorem implies a $K_k$-decomposition of $G^\ast$. Indeed, by our choice of $t$ we have that $(k-1) | (n+t-1)$ and ${k \choose 2} | {n+t \choose 2}$. Since $G^\ast$ is obtained from $K_{n+t}$ by removing edge disjoint $K_k$'s, its degrees are still divisible by $k-1$ and its number of edges is divisible by ${k \choose 2}$.
The decomposition of $G^\ast$, together with the $K_k$'s from $\mathcal{S}_1,\ldots \mathcal{S}_{|B|}$, forms a full $K_k$-decomposition of $G'$.
\end{proof}

We now present a construction showing that Lemma~\ref{lem:kkdecomp} is tight for $r > 2n$  up to a multiplicative constant, which also translates to sharpness of Theorem \ref{thm:steiner} and Theorem \ref{thm:kkdecomp}. To remind the reader, a linear lower bound on $r$ is not a coincidence: for $k = 3$ and $r < n/2$, Horsley~\cite{horsley2014embedding} conjectured that any partial STS of order $n$ has an embedding of the same order, and it is natural to believe that a similar result should hold for partial $(n,k)$-designs as well. Let $k \ge 3$ and $2n < r \le 4n^2/k^2$. We define copies of $K_k$ to be deleted as follows: Fix a vertex $v$ and a subset $V'$ of size $k \sqrt{r} / 2$ (the upper bound on $r$ tells us that $|V'| \leq n$, thus we can indeed choose such $V'$). Take a family of $K_k$'s given by an arbitrary $K_k$-decomposition of $K_n[V']$ together with copies of $K_k$ obtained by joining $v$ to the copies of $K_{k-1}$ of an arbitrary $K_{k-1}$-factor of $K_n \setminus (V' \cup \{v\})$. All these copies of $K_k$ are clearly edge-disjoint and there are less than $r/2 + n \le r$ of them. Let $G$ be the graph obtained by removing these copies. Then $V'$ is an independent set in $G$ and $v$ is not connected to any vertex outside of $V'$. In particular, for every $w \in V'$, in order to cover the edge $\{v, w\}$ by a copy of $K_k$ one needs to add  $k-2$ new vertices. Moreover, all these sets of new vertices have to be disjoint in order to keep the copies of $K_k$'s edge-disjoint, resulting in $(k-2) |V'| = k(k-2)\sqrt{r} / 2$ new vertices. 

As remarked earlier, we cannot always guarantee that a partial $(n,k)$-design $\F$ always has an embedding of the same order. The next corollary of Lemma \ref{lem:kkdecomp} shows that we can find a partial $(n,k)$-design $\F'$ over the same set such that $\F \subseteq \F'$ and $\F'$ is close to being a complete $(n, k)$-design.

\begin{corollary}
For every $k \ge 3$ there exists $\eps, n_0 > 0$ such that the following holds. Given a partial $(n, k)$-design $\F$ of order $n \ge n_0$ with $|\F| \le \eps n^2$ blocks, there exists a partial $(n, k)$-design $\F'$ over the same set such that $\F \subseteq \F'$ and $\F'$ covers all but at most $21 k^3 \sqrt{|\F|} n$ pairs of vertices.
\end{corollary} 
\begin{proof}
Each block of $\F$ corresponds to a copy $K_k$ in $K_n$, thus let $G$ be the graph obtained from $K_n$ by deleting these $|\F|$ copies. Then, for some $t \le 7 k^2 \sqrt{|\F|}$, $G \ast K_t$ contains a $K_k$-decomposition. Note that every two $K_k$'s in a $K_k$-decomposition which contain the same vertex $v$ intersect only in $v$. Thus there are exactly $(t + n - 1) / (k-1)$ copies of $K_k$ which contain a vertex $v$, thus at most
$$
    t \cdot \frac{t + n - 1}{k-1} \le \frac{2tn}{k-1} \le 21 k n \sqrt{|\F|}
$$
copies of $K_k$ contain at least one newly added vertex. The first inequality holds for sufficiently small $\eps$. Therefore, deleting all the copies of $K_k$ containing at least one newly added vertex gives a partial $(n,k)$-design over $[n]$ which has at most $21k^3 n \sqrt{|\F|}$ pairs of vertices uncovered.
\end{proof}

We finish this section with a short discussion on possible extensions of Theorems~\ref{thm:steiner} and~\ref{thm:kkdecomp} to hypergraphs. Denote by $K_n^{(3)}$ the complete 3-uniform hypergraph with $n$ vertices. Somewhat surprisingly, in this case is not true that by removing $\eps n^3$ edge-disjoint copies of $K_4^{(3)}$ from $K_n^{(3)}$ we can always decompose the remaining hyperedges into copies of $K_4^{(3)}$, even if we add linearly many full degree vertices. The following construction shows this: Fix two vertices $v,w$ and choose a collection $\F$ of quadratically many hyperedges in $V(K_n^{(3)})\setminus \{v,w\}$ that cover every pair of vertices exactly once. In other words, $\F$ is a STS on $V(K_n^{(3)}) \setminus \{v, w\}$. Remove all $K_4^{(3)}$'s formed by a hyperedge in $\F$ together with the vertex $v$. Note that this removes only $O(n^2)$ hyperedges, whereas $K_n^{(3)}$ has $\Theta(n^3)$ edges. Now consider the hyperedges containing $v,w$. Each hyperedge $\{v,w,x\}$ needs to be covered, but the $\{v,x,y\}$ edge has already been used for all $y$ and so a new vertex $x'$ needs to be added, moreover these new vertices need to be different and so at least $n-2$ new vertices need to be added. 

Note, however, that by a celebrated result of Keevash~\cite{keevash} about existence of designs in hypergraphs with very large degree, there exists an absolute constant $C$ such that for $t = Cn + O(1)$ we have that $\HH \ast K_t$ contains a $K_4^{(3)}$-decomposition for any hypergraph $\HH$ on $n$ vertices. To salvage a deficiency problem for hypergraphs, we believe the following is plausible.

\begin{question}
Is it true that for every $\eps>0$ there exists a $\delta>0$ such that the following holds for large enough $n$: Let $\HH$ be obtained from $K_n^{(3)}$ by deleting some number of edge-disjoint copies of $K_4^{(3)}$ so that every vertex is incident to at most $\delta n^2$ such deleted copies. Then there exists an integer $t\leq \eps n$ so that $\HH\ast K_t$ has a $K_4^{(3)}$-decomposition.
\end{question}


\section{Embedding partial Latin squares}\label{sec:latin}

In this section we prove Theorem \ref{thm:latinort}. As mentioned earlier, Theorem \ref{thm:latinort} can be seen as a multi-partite version of Theorem \ref{thm:kkdecomp}. Let us make this connection explicitly. Suppose we are given $k-2$ pairwise orthogonal partial Latin squares $P_1, \ldots, P_{k-2}$ of order $n$. Let $V_1, \ldots, V_k$ be vertex classes of $K_{n, \ldots, n}$, the complete $k$-partite graph with each part being of size $n$ (we call such a complete $k$-partite graph \emph{$n$-balanced}), and label vertices in each $V_i$ as $\{1, \ldots, n\}$. Throughout this section we are always working with $k$-partite graphs, thus there is no risk of ambiguity by writing $K_{n, \ldots, n}$. Slightly abusing notation, we implicitly differentiate between vertices with the same label coming from different $V_i$'s. We now form a family of edge-disjoint cliques in $K_{n, \ldots, n}$ as follows: for each $i, j \in [n]$ such that $(P_t)_{i,j} \neq \ast$ for at least one $t \in \{1, \ldots, r-2\}$, take a clique with the vertex set
$$
     \{ i \in V_1\} \cup \{j \in V_2\} \cup \{ w \in V_{t+2} \colon 1 \le t \le k-2 \text{ and } (P_{t})_{i,j} = w\}.
$$
Clearly, each such clique is of size at least $3$ and at most $k$. Let us denote all these cliques by $A_1, \ldots, A_m$. It is important to observe that every such clique contains a vertex from $V_1$ and $V_2$. Note that if there exists a $K_k$-decomposition of a complete $k$-partite graph $K_{n', \ldots, n'}$, for some $n' \ge n$, such that each $A_i$ belongs to a clique from this decomposition, then there exist pairwise orthogonal Latin squares $P_1', \ldots, P_{k-2}'$ of order $n'$, with each $P_i'$ being an embedding of $P_i$: simply set $(P'_t)_{i,j}$ to be the label of the vertex in $V_{t+2}$ which belongs to the $K_k$ from this decomposition which contains $i \in V_1$ and $j \in V_2$.

In the case where all $A_i$'s are of size $k$ then what we are asking for is just an embedding of a partial $K_k$-decomposition (that is, its multi-partite version). However, some cliques could be smaller than $k$ and handling this is one of the main differences compared to the proof of Theorem \ref{thm:kkdecomp}. This is done in the following lemma. 

\begin{lemma}\label{lem:extend}
For every $k \ge 3$ there exists $\varepsilon > 0$ such that the following holds. Let $m \le \eps n^2$ and $A_1,A_2,\ldots,A_m$ be edge-disjoint cliques in $K_{n, \ldots, n}$ such that each $A_i$ contains a vertex from $V_1$ and a vertex from $V_2$. Then, for $n' = n +  8 k \sqrt{m}$, there exist a collection of edge-disjoint $K_k$'s $B_1,B_2,\ldots, B_m$ in $K_{n', \ldots, n'}$ such that $A_i\subseteq B_i$ for all $i \in [m]$.
\end{lemma}
\begin{proof}
Let us denote the set of newly added $8k\sqrt{m}$ vertices to each $V_i$ by $T_i$. It will be convenient for the proof to further split each $T_i$ into two sets, say $T_i'$ and $T_i''$, of (nearly) equal size.

Let $Q = \{v_1,v_2,\ldots,v_{\sqrt{m}}\}$ be a set of \emph{bad} vertices, defined iteratively as follows. Having defined $v_1,v_2,\ldots,v_i$, let $v_{i+1} \in V(K_{n, \ldots, n}) \setminus \{v_1, \ldots, v_i\}$ be a vertex incident to the largest number of $A_j$'s that are not incident to any of the vertices $v_1,v_2,\ldots,v_i$. Observe that, by definition, any $v\not\in Q$ is incident to at most $\sqrt{m}$ of $A_j$'s that do not contain a vertex from $Q$.

Next, we say that an $A_i$ is \emph{bad} if it contains at least two vertices from $Q$. Since $A_i$'s are edge-disjoint, this implies that every vertex is incident to at most $|Q| = \sqrt{m}$ bad cliques. By relabelling $A_i$'s, we may assume $A_1,\ldots,A_s$ are bad and $A_{s+1},\ldots,A_m$ are not bad, for some $s$. We will first extend the bad cliques using sets $T_i'$ and then extend the remaining cliques using sets $V_i \cup T_i''$. By using such disjoint sets, we can treat both cases independently.

\noindent
\textbf{Extending bad cliques.} Suppose we have extended $A_1,\ldots,A_{i-1}$ to edge-disjoint copies $B_1, \ldots, B_{i-1}$ of $K_k$ using only vertices from sets $T_i'$. Without loss of generality, we may assume $A_i$ is a clique on $v_1 \in V_1, v_2\in V_2,\ldots, v_z \in V_z$, for some $3 \le z < k$ (if $z = k$ then $A_i$ is already a copy of $K_k$). We iteratively extend $A_i$ to a copy of $K_{z+1}$ using a vertex from $T'_{z+1}$, then to a copy of $K_{z+2}$ using a vertex from $T'_{z+2}$, and so on.

Since every vertex is incident to at most $\sqrt{m}$ bad cliques, at most $\sqrt{m}$ vertices from $T'_{z+1}$ are in the same clique as $v_1$ so far. The same holds for $v_2, v_3, \ldots, v_z$. In particular, all but at most $z \sqrt{m}$ vertices in $T'_{z+1}$ are such that together with $A_i$ they form a copy of $K_{z+1}$ which is edge-disjoint from all previously obtained $K_k$'s. Let $t_{z+1} \in T'_{z+1}$ be an arbitrary such vertex. Note that every clique so far which contains $t_{z+1}$ also contains a vertex from $Q$, thus from edge-disjointness we have that $t_{z+1}$ is contained in at most $|Q| = \sqrt{m}$ cliques. Continuing the process, all but at most $z \sqrt{m} + \sqrt{m}$ vertices in $T'_{z+2}$ do not appear in a same clique with either of $v_1, \ldots, v_z, t_{z+1}$. In general, after extending $A_i$ to a copy $K_{z+j}$, for some $j < k - z$, all but at most $(z+j)\sqrt{m}$ vertices in $T'_{z+j+1}$ are `available'. Therefore, we can repeat this process until $A_i$ is extended to a copy of $K_k$ which is edge-disjoint from all other cliques.



\noindent
\textbf{Extending good cliques.} 
Throughout the process of extending good cliques, we maintain an invariant that every vertex in $V_i \cup T''_i$ which is not bad belongs to at most $2\sqrt{m}$ cliques. Note that at the beginning of the procedure this condition is satisfied: a vertex which is not bad is contained in at most $\sqrt{m}$ cliques that do not contain a vertex from $Q$ (otherwise such a vertex would be bad) and at most $|Q| = \sqrt{m}$ cliques which contain a vertex from $Q$, owing to all the cliques being edge-disjoint. 

Suppose we have extended cliques $A_{s+1}, \ldots, A_{i-1}$ to $B_{s+1}, \ldots, B_{i-1}$ and consider a good clique $A_i$. Without loss of generality, we may assume that it is a clique on $v_1 \in V_1, \ldots, v_z \in V_z$, for some $z < k$, such that $v_1$ is a bad vertex (if it has a bad vertex at all). By the definition of a good clique, neither of $v_2, \ldots, v_z$ can then be bad. As each clique uses a vertex from $V_1$ and $V_2$, $v_1$ is incident to at most $n$ cliques (note that for this it is crucial that every clique contains a vertex from both $V_1$ and $V_2$; otherwise our `without loss of generality' assumption would not be true). Therefore, by the invariant, at most $n + 2 z \sqrt{m}$ vertices in $V_{z+1} \cup T''_{z+1}$ are part of a clique together with one of $v_1, \ldots, v_z$. Therefore there are at least $k \sqrt{m}$ `available' vertices. Let us choose $v_{z+1} \in V_{z+1} \cup T''_{z+1}$ to be one such available vertex which is used the least number of times (that is, it belongs to the smallest number of cliques $B_1, \ldots, B_{i-1}, A_i, A_{i+1}, \ldots, A_m$). This implies that $v_{z+1}$ appears in at most
$$
    \frac{m}{k \sqrt{m}} < \sqrt{m}
$$
cliques. In particular, by extending $A_i$ using $v_{z+1}$ we have that $v_{z+1}$ appears in at most $\sqrt{m} + 1$ cliques, thus the invariant remains satisfied.

Continuing this process, we have that there are at most $n + 2(z+1) \sqrt{m}$ vertices in $V_{z+2} \cup T''_{z+2}$ which appear in a clique with either $v_1, \ldots, v_z, v_{z+1}$. Again, by choosing $v_{z+2} \in V_{z+2} \cup T''_{z+2}$ to be an available vertex which appears in the least number of cliques, we maintain the invariant on the number of cliques which contain any vertex from $V_{z+2} \cup T''_{z+2}$. In general, after extending $A_i$ to a copy of $K_{z+j}$, for some $j$ such that $z + j < k$, we have at least $k \sqrt{m}$ available vertices in $V_{z+j+1} \cup T''_{z+j+1}$, thus we can continue the process until $A_i$ is extended to a copy of $K_k$.
\end{proof}

The following lemma plays the role analogue to the role of Lemma \ref{lem:defHajSze} in the proof of Theorem \ref{thm:kkdecomp}.

\begin{lemma} \label{lemma:defLatinortHSz}
Let $k \ge 3$ be an integer and let $G$ be a $k$-partite graph with vertex classes $V_i = S_i\dot\cup T_i$, for $i= 1, 2, \ldots, k$, where $|T_i| \ge 9 k \sqrt{m}$ and all $V_i$'s are of the same size. Moreover, suppose that every vertex $v \in V_i$ has at least $|V_j| - \sqrt{m}$ neighbors in every $V_j$, for $i \neq j$. Then no matter how we remove at most $r_v$ edges between $v \in S_i$ and each $S_j$, for $i \neq j$, such that $\sum_{v \in S_i} r_v \le m$ for each $i$, the resulting graph contains a $K_k$-factor.
\end{lemma}

As a clarification, note that we allow $r_v$'s to be different for different vertices as long as the sum condition is satisfied. Observe that the bound on $T_i$ cannot be improved: let $|S_i|=\sqrt{m}$  and $|T_i| = (k-1)\sqrt{m} - 1$ for all $i$, and suppose $G$ is the complete $r$-partite graph. Remove all the edges between every two sets $S_i$ and $S_j$. Every $K_k$ in the resulting graph which contains a vertex from, say, $S_i$, must contain a vertex from $T_j$ for every $j \neq i$. In particular, in order to cover all $\bigcup_{i \ge 2} S_i$ with vertex-disjoint copies of $K_k$ we need $|T_1| \ge (k-1) \sqrt{m}$.

In the proof of Lemma \ref{lemma:defLatinortHSz} we use the following multi-partite version of the Hajnal--Szemer\'edi theorem due to Fischer \cite{fischer1999variants}.

\begin{lemma} \label{lemma:fischer}
    Let $G$ be a $k$-partite $n$-balanced with vertex classes $V_1, \ldots, V_k$, such that each vertex $v \in V_i$ has at least $\frac{2k - 3}{2k - 2}n$ neighbors in each $V_j$, for $j \neq i$. Then $G$ contains a $K_k$-factor.
\end{lemma}

It is worth noting that, unlike the Hajnal--Szemer\'edi theorem, owing to a large minimum partite-degree, Lemma \ref{lemma:fischer} is very easy to prove using a straightforward matchings-based argument. A better (and optimal) bound on the minimum partite-degree was obtained by Keevash and Mycroft \cite{keevash2015multipartite} and Lo and Markstr{\"o}m \cite{lo2013multipartite}. However, a drawback in their results is that they require $n$ to be sufficiently large with respect to $k$, whereas in our proof we rely on the fact that Lemma \ref{lemma:fischer} holds for all $n$.

\begin{proof}[Proof of Lemma \ref{lemma:defLatinortHSz}]
Let $S:=\bigcup_i S_i$ and $T:=\bigcup_i T_i$, let $E$ be the set of removed edges and $G'$ the resulting graph. Let us call a vertex $v$ \emph{bad} if $r_v > \sqrt{m}$. Note that every $S_i$ contains at most $\sqrt{m}$ bad vertices, and there are at most $k \sqrt{m}$ bad vertices all together. Let us denote the set of these vertices by $Q$ and set $p := |Q|$. We first find a family of vertex-disjoint $K_k$'s covering all the vertices in $Q$.

To this end, consider an arbitrary ordering $b_1, \ldots, b_p$ of the vertices in $Q$. Suppose that for each $j\leq i$, for some $i < p$, we have found a set $B_j \subset T\cap N_{G'}(b_j) = T \cap N_{G}(b_j)$ of $k-1$ vertices such that $G[B_j]\cong K_{k-1}$ and $B_{j_1}\cap B_{j_2}=\emptyset$ for each $1\leq j_1<j_2\leq i$ (recall that every edge with one endpoint in $T$ which is present in $G$ is also present in $G'$).  Set $P_{j}^i:= T_j \setminus\left( \bigcup_{t\leq i} \{b_t\} \cup B_t \right)$ and note that 
$$
    |P_j^i| = |T_j| - i \ge  8k \sqrt{m}.
$$
Since in $G$ every vertex was missing at most $\sqrt{m}$ edges into any other vertex class, we have 
$$
    |N_G(b_{i+1})\cap P_{j}^i|\geq |P_{j}^i|- \sqrt{m}\geq 7 k \sqrt{m}.
$$ 
For every $j$ such that $b_{i+1} \notin V_j$, take an arbitrary subset $U^i_j\subseteq N(b_{i+1}\cap P_{j}^i)$ of size $7k\sqrt{m}$ and let $H_i$ be a subgraph of $G$ induced by $\bigcup U_j^i$, where the union goes over all $j$ such that $b_{i+1} \notin V_j$. Every vertex $v \in U_{j}^i$ has at least
$$
    |U_{j'}^i| - \sqrt{m} \ge (1 - 1/k) |U_{j'}^i|
$$
neighbors in $U_{j'}^i$, thus a simple greedy argument shows that $H_i$ contains a copy of $K_{k-1}$. Let $B_{i+1}$ be the vertex set of such $K_{k-1}$. As $B_{i+1}$ lies in the neighborhood of $b_{i+1}$, the two together form a copy of $K_k$. Continuing this process, we have found a collection of vertex-disjoint $K_k$'s covering every vertex of $Q$.

Consider the graph $H':= G' \setminus \left( \bigcup_{i = 1}^p \{b_i\} \cup B_i\right)$. Then $H'$ is a $k$-partite graph on the vertex set $W_1, \ldots, W_k$ and, as each copy of $K_k$ in $G$ spans across all $k$ vertex classes, we trivially have that all $W_i$'s are of the same size, and $|W_i| \ge |T_i| - p \ge 8 k \sqrt{m}$. Moreover, by the choice of $Q$ and the degree assumptions on $G$, each vertex $v \in W_i$ has at least
$$
    |W_j| - \sqrt{m} - \sqrt{m} \ge (1 - 1/(2k)) |W_j|
$$
neighbors in each $W_j$. Hence, by Theorem \ref{lemma:fischer}, $H'$ has a $K_k$-factor. This $K_k$-factor together with the copies of $K_k$ which cover $Q$ forms a $K_k$-factor of $G'$.
\end{proof}

Following the connection between $K_k$-decompositions and orthogonal Latin squares from the beginning of the section, it is clear that the following lemma implies Theorem~\ref{thm:latinort}.

\begin{lemma} \label{lemma:latinortmain}
For every $k \ge 3$ there exist $\eps > 0$ such that the following holds for every sufficiently large $n$. Given at most $m \le \eps n^2$ edge-disjoint cliques $A_1, \ldots, A_m$ in $K_{n, \ldots, n}$, each of which contains a vertex from the first two vertex classes, there exists a $K_k$-decomposition of $K_{N, \ldots, N}$, where $N = n + 20 k \sqrt{m}$, such that each $A_i$ is contained within a $K_k$ from such a decomposition.
\end{lemma}

In the proof of Lemma \ref{lemma:latinortmain}, instead of Gustavsson's theorem (Theorem \ref{thm:gust}), we use the following multi-partite version.

\begin{theorem}[Corollary of \cite{barber}] \label{thm:gust_multip}
For every integer $k \ge 3$ there exists $\gamma > 0$ such that the following holds for every sufficiently large $n$. If $G$ is a $k$-partite $n$-balanced graph with vertex classes $V_1, \ldots, V_k$ such that every vertex $v \in V_i$ has $d_v \ge (1 - \gamma)n$ neighbors in each $V_j$, for $i \neq j$, then $G$ has a $K_r$-decomposition.
\end{theorem}

For clarity, note that a vertex $v$ has the same number of neighbors, $d_v$, in each $V_i$ (other than its own vertex class) and that thus number might be different for different vertices. The main result in \cite{barber} is more general and here we have stated a streamlined version which suffices for our application.

\begin{proof}[Proof of Lemma \ref{lemma:latinortmain}]
First, by  Lemma~\ref{lem:extend} we have that there exists a family of $m$ edge-disjoint $K_k$'s $B_1, \ldots, B_m$ in $K_{n', \ldots, n'}$, where $n' = n + 9 k \sqrt{m}$, such that each $A_i$ is contained within $B_i$. Let us denote the vertex classes of $K_{n', \ldots, n'}$ by $V_1, \ldots, V_k$, and let $E$ denote the edge set of all $B_i$'s. Note that $E$ contains exactly $m$ edges between every two sets $V_i$ and $V_j$, and moreover the number of edges from $E$ between $v \in V_i$ and $V_j$ is the same as between $v$ and $V_{j'}$, for $i \neq j, j'$ and any vertex $v \in V_i$. Let us denote this number by $r_v$ (note that different vertices might have different values of $r_v$). This will allow us to eventually apply Lemma \ref{lemma:defLatinortHSz}.

We show the lemma for $N = n' + 11 k \sqrt{m}$. We can think of $K_{N, \ldots, N}$ as taking $K_{n', \ldots, n'}$ and adding $11 k \sqrt{m}$ new vertices, denoted by $T_i$, to each $V_i$ and connecting them completely to all other vertex classes (including all other $T_j$'s). Let $G$ be a graph obtained from $K_{N, \ldots, N}$ by removing edges in $E$. Note that it suffices to show that $G$ has a $K_k$-decomposition. 

We say that a vertex $v$ in $G$ is \emph{bad} if $r_v > k \sqrt{m}$. There are at most $\sqrt{m} / k$ bad vertices in each $V_i$, thus by arbitrarily nominating some additional vertices to be bad, we can assume that each $V_i$ contains exactly $\sqrt{m} / k$ bad vertices. This also gives $\sqrt{m}$ bad vertices overall. We iteratively build a collection of edge-disjoint copies of $K_k$ in $G$ that cover every edge incident to at least one  bad vertex. Let $b_1, \ldots, b_{\sqrt{m}}$ be an arbitrary ordering of the bad vertices and suppose we have defined collections $\mathcal{S}_1,\mathcal{S}_2,\ldots,\mathcal{S}_i$ of distinct copies of $K_r$ in $G'$, for some $i < \sqrt{m}$, such that the following holds:
\begin{enumerate}
    \item\label{item:disj2} every distinct $F_1 \in \mathcal{S}_{j_1}$ and $F_2 \in \mathcal{S}_{j_2}$, $1 \le j_1, j_2 \le i$, are  edge-disjoint, 
    \item\label{item:isol2} for every $1\leq j\leq i$,  every edge in $G'$ incident to $b_j$ belongs to some copy of $K_k$ in $\mathcal{S}_j$ and every copy of $K_k$ in $\mathcal{S}_j$ contains $b_j$.
\end{enumerate}
Set $G_i$ to be the graph obtained from $G$ by deleting all edges in the $K_k$'s from $\bigcup_{j\leq i}\mathcal{S}_j$. By~(\ref{item:isol2}) we have that the vertices $b_1,b_2,\ldots,b_i$ are isolated in $G_i$. By~(\ref{item:disj2}) we have that for every $1\leq j \leq i$ every vertex $v\neq b_j$ occurs in at most one of the $K_k$'s in $\mathcal{S}_j$ (as every such $K_k$ also contains $b_j$). Hence by forming $G_i$ from $G$, the number of neighbors of every vertex $v \in (V_i \cup T_i) \setminus \{b_1,b_2,\ldots,b_i\}$ in each $V_j \cup T_j$, $j \neq i$, reduces by at most $i < \sqrt{m}$. As in $G$ every vertex $v \in V_j \cup T_j$ was connected to every vertex in $T:=\cup_{j' \neq j}  T_{j'}$, letting $N_i(v)$ denote the neighborhood of a vertex $v$ in $G_i$, we have
\begin{equation} \label{eq:Tji}
    |N_i(b_{i+1}) \cap T_j|\geq 11 k \sqrt{m} - \sqrt{m}
\end{equation}
for every $j \neq z$, where $z \in [k]$ is such that $b_{i+1} \in V_{z}$. Next, for each $j \neq z$ set $T_j^i = N_i(b_{i+1})\cap T_j$ and $S_j^i:=N_i(b_{i+1}) \cap V_j$. As all the missing edges in $G_i$ come from edge-disjoint copies of $K_k$, we have that all the sets $S_j^i$ are of the same size, as well as all of $T_j^i$. Let $H_i$ be a $(k-1)$-partite subgraph of $G_i$ induced by $\bigcup_{j \neq z} T_j^i \cup S_j^i$, with edges from $E$ temporarily added back in. Therefore, every vertex $v \in S_j^i \cup T_j^i$ has at least
$$
    |S_{j'}^i \cup T_{j'}^i|- \sqrt{m}
$$
neighbors into every $S_{j'}^i \cup T_{j'}^i$, for $j \neq j'$ and $j, j' \neq z$. By the discussion on the properties of $E$ from the beginning of this proof, we may remove edges from $E$ again and apply Lemma~\ref{lemma:defLatinortHSz} with $k$ playing the role of $k-1$ to find a collection $\mathcal{F}$ of vertex-disjoint $K_{k-1}$'s covering every vertex of $H_i$. Note that this is indeed possible as every vertex $v \in S_j^i$ is incident to at most $r_v$ edges from $E$ with the other endpoint in $S_{j'}^i$ and $\sum_{v \in S_j^i} r_v \le \sum_{v \in V_j} r_v \le m$. By adding $b_{i+1}$ to every copy of $K_{k-1}$ in $\mathcal{F}$, we obtain a desired collection $\mathcal{S}_{i+1}$ of  edge-disjoint $K_k$'s. This takes care of bad vertices.

Let $G'$ be the graph obtained from $G$ by deleting all edges in the $K_k$'s from $\bigcup_{j\leq p}\mathcal{S}_j$, and moreover delete the isolated vertices $b_1,b_2,\ldots,b_{\sqrt{m}}$ from $G'$. As before, by forming $G'$ from $G$, the degree of every vertex $v \in (V_i \cup T_i) \setminus \{b_1,b_2,\ldots,b_{\sqrt{m}}\} =: W_i$ into $W_j$, for $j \neq i$, reduced by at most $\sqrt{m}$. Hence, as a vertex $v \in W_i$ is not bad, it has at least
$$
    |W_j| - \sqrt{m} - r_v \ge |W_j| - 2k \sqrt{m}
$$
neighbors in every $W_j$, for $j \neq i$. Moreover, as every $V_j$ contains the same number of bad vertices (which is $\sqrt{m}/k$), we have that all $W_j$'s are also of the same size, namely $|W_j| \ge n' + 10 k \sqrt{m} > n$. Note that all the missing edges in $G'$ correspond to edge-disjoint $K_k$'s, thus $G'$ satisfies the conditions of Theorem \ref{thm:gust_multip}. Since for $\eps > 0$ sufficiently small compared to $k$ we have 
$$
    2k\sqrt{m} \le \gamma n \le \gamma |W_j|,
$$
Theorem \ref{thm:gust_multip} implies that $G'$ contains a $K_k$-decomposition. This decomposition together with the $K_k$'s in $\mathcal{S}_1,\ldots \mathcal{S}_{\sqrt{m}}$, forms a $K_k$-decomposition of $G$.
\end{proof}

We now show that Lemma \ref{lemma:latinortmain} is optimal up to a multiplicative constant factor for $m \ge 2n$. Consider a complete $k$-partite $n$-balanced graph $K_{n, \ldots, n}$ with vertex classes $V_1, \ldots, V_k$. From each $V_i$ pick an arbitrary subset $X_i \subseteq V_i$ of size $\sqrt{m/2}$, and choose a vertex $v_1 \in V_1 \setminus X_1$. We now form a family of at most $m$ cliques as follows. First, take a $K_k$-decomposition of the subgraph induced by $X_1, \ldots, X_k$ (giving altogether $m/2$ cliques). Second, take a $K_{k-1}$-factor of the $(k-1)$-partite subgraph induced by $V_2 \setminus X_2, \ldots, V_k \setminus X_k$, and to each copy of $K_{k-1}$ in this factor append $v_1$ (giving additional $n - |X_i| \le m/2$ cliques). Suppose that $K_{n', \ldots, n'}$, for some $n' \ge n$, contains a $K_k$-decomposition which contains all of these cliques. Observe that a copy of $K_k$ in this decomposition which contains $v_1$ and any $v_2 \in X_2$, requires a new vertex in each of the $k-2$ other vertex classes. In particular, it requires a new vertex in the $k$-th class, and all these new vertices have to be distinct. The same holds for a copy of $K_k$ which contains $v_1$ and $v_3 \in X_3$, and so on. As each new vertex has to be different, this requires at least $(k-2)|X_i| = (k-2)\sqrt{m/2}$ new vertices in the $k$-th color class. Finally, in order for a complete $k$-partite graph to have a $K_k$-decomposition, all vertex classes have to be of the same size.

Similarly as it was the case in Section \ref{sec:designs}, it is likely that if we only have $m < cn$ cliques, for some small constant $c > 0$, then $n' = n$ suffices.



\section{Hamilton cycles}\label{sec:hami}

To prove Theorem \ref{thm:hami} we will need the following well known result of Chv\'atal~\cite{chvatal} which gives a sufficient condition for a graph to contain a Hamilton cycle.
\begin{theorem}\label{thm:chvatal}
Let $G$ be a graph with vertex degrees $d_1 \leq d_2 \leq \ldots \leq d_n$, where
$n \geq 3$. If $d_i > i$ or $d_{n-i} \geq n - i$ for each $1 \le i < n/2$, then $G$ is Hamiltonian.
\end{theorem}
\begin{proof}[Proof of Theorem~\ref{thm:hami}] Let $G$ be an $n$-vertex graph and $t$ an integer so that $G\ast K_t$ does not contain a Hamilton cycle. Let $m(G)$ denotes the number of \emph{missing} edges in $G$, that is $m(G) = \binom{n}{2} - e(G)$. To prove the first part of the theorem, it suffices to show the following bounds on $m(G)$:
\begin{itemize}
    \item If $n + t$ is even then
    \begin{equation} \label{eq:mG1}
        m(G) \ge \begin{cases}
            t(n-1) - \binom{t}{2} & \text{ if } t \le (n+4)/5 \\
            \binom{\frac{n+t+2}{2}}{2} - 1 &\text{ if } t \ge (n+4)/5.
        \end{cases}
    \end{equation}
    
    \item If $n + t$ is odd then
    \begin{equation} \label{eq:mG2}
        m(G) \ge \begin{cases}
            t(n-1) - \binom{t}{2} & \text{ if } t \le (n+1)/5 \\
            \binom{\frac{n+t+1}{2}}{2} &\text{ if } t \ge (n+1)/5.
        \end{cases}
    \end{equation}\textbf{}
\end{itemize}

Let $G' = G \ast K_t$ and label the vertices of $G'$ as $v_1, \ldots, v_{n + t}$ in a non-decreasing order with respect to the degree. Since $G'$ is not Hamiltonian, it does not satisfy the conditions of Theorem~\ref{thm:chvatal}. Note that the minimum degree of $G'$ is at least $t$, thus there exists some $t \le i< (n+t)/2$  such that $d_i \leq i$ and $d_{n+t-i}< n+t-i$. Denote $S:=\{v_1,v_2,\ldots,v_i\}$. From $d_1 \le d_2 \le \ldots \le d_i \leq i$ we deduce that the number of edges missing from $G'$ is at least 
$$
    m(G') \ge \sum_{j = 1}^i (n + t - 1 - d_j) - \binom{i}{2} \ge i(n+t-1-i) - \binom{i}{2} =: f(i).
$$
It is important to notice that $m(G') = f(i)$ iff all the missing edges are incident to $S$ and $S$ is an independent set. Moreover, we have $m(G') = m(G)$. For brevity, let us set $u = (n+t)/2 - 1$ if $n + t$ is even, and $u = (n+t-1)/2$ otherwise. As $f(\cdot)$ is a downward facing parabola and $t \le i \le u$, we have 
\begin{equation} \label{eq:boundmG}
    m(G) \ge \min\{ f(t), f(u) \}.
\end{equation}
Again, if $i \not \in \{t, u\}$ then we have a strict inequality. It is now a matter of straightforward calculation to show that the bound from \eqref{eq:boundmG} gives precisely \eqref{eq:mG1} and \eqref{eq:mG2}. In particular, if $n + t$ is even then $f(t) \le f(u)$ iff $t \le (n+4)/5$, and if $n + t$ is odd then $f(t) \le f(u)$ iff $t \le (n+1)/t$.

We now derive extremal constructions, that is we describe how $G$ has to look in order to achieve equality in \eqref{eq:boundmG}. First, recall the properties obtained so far: every missing edge in $G$ is incident to $S$ and $S$ is an independent set; $i \in \{t, u\}$. Let us first consider the case $f(t) \le f(u)$ and $i = t$. Note that this corresponds to $t \le (n+4)/5$ if $n+t$ is even, and $t \le (n+1)/5$ if $n+t$ is odd. In either case we have 
$$
    m(G) = f(t) = t(n - 1) - \binom{t}{2},
$$
which implies that all the edges incident to $S$ are missing. Therefore $G$ is isomorphic to a graph which contains exactly $t$ isolated vertices.

The other case is a bit more involved and we have to consider the cases when $n + t$ is even and odd separately. Let us first assume $n + t$ is even. Then
$$
    m(G) = f(u) = f((n+t)/2 - 1) = \binom{\frac{n+t+2}{2}}{2} - 1 = \binom{u+2}{2} - 1 = \binom{u}{2} + 2u.
$$
As we know that every edge within $S$ is missing, from the previous bound we conclude that there have to be exactly $2u$ edges missing between $S$ and $V(G) \setminus S$. We aim to show that these $2u$ edges have to be incident to exactly two vertices in $V(G) \setminus S$, from which we conclude that $G$ is isomorphic to the graph obtained from $K_n$ by choosing a set $S$ of $u+2$ vertices and removing all but one edge within $S$. To conclude that there are exactly $2$ vertices incident to all of these $2u$ missing edges we employ, for the first time, the second part of Chv\'atal's theorem: $d_{n + t - u} < n + t - u$. Recall that $u = (n+t)/2 - 1$, thus
$$
    d_{u+1} \le d_{u+2} \le (n+t)/2.
$$
Therefore both $v_{u+1}$ and $v_{u+2}$ are missing at least $n + t - 1 - (n+t)/2 = u$ edges each, and all these edges have to be also incident to $S$. Therefore we obtain the remaining $2u$ missing edges. Note that all vertices in $V(G) \setminus (S \cup \{v_{u+1}, v_{u+2}\})$ have full degree.

Let us finally consider the case $i = u$ and $n + t$ is odd. We now have $u = (n+t-1)/2$ and 
$$
    m(G) = f(u) = \binom{\frac{n + t + 1}{2}}{2} = \binom{u}{2} + u.
$$
Similarly to the previous case, we aim to show that all of the $u$ missing edges with one  endpoint in $V(G) \setminus S$ have to be incident to a single vertex. From $d_{n + t - u} \le n + t - u - 1$ and $n + t - u = u + 1$ we conclude that $v_{u+1}$ is missing at least $u$ edges. Moreover, all these missing edges have another endpoint in $S$. As every other vertex has full degree, we conclude that $G$ is isomorphic to a graph obtained from $K_n$ by choosing a set $S$ of $u+1$ vertices and removing all edges within $S$.

\end{proof}

\section{Triangle factors}\label{sec:corrhaj}

The main idea of the proof of Theorem~\ref{thm:corrhajdef} is to first deal with vertices with small degree and then show that the remaining graph satisfies the minimum degree condition of the Corr\'adi-Hajnal theorem. The details are somewhat tricky and require verification of several technical inequalities.

\begin{proof}[Proof of Theorem \ref{thm:corrhajdef}]
Before we start with the proof let us fix the following constants: $\alpha =0.002$, $\beta  = 0.011$ and $\gamma =0.05$. These values are chosen so that all following inequalities work. Moreover, it will be convenient to set 
$$ 
    \ell = 0.1 n \qquad \text{ and } \qquad h = 0.9n.
$$

Suppose that $n$ is sufficiently large and $t \le n / 1000$ is odd. Let $k =  \lceil (t+1) / 2 \rceil$ and let $G$ be an $n$-vertex graph which is missing at most
\begin{equation} \label{eq:missing}
    d := \begin{cases}
        \binom{k}{2} + k(n - k) - 1, &\text{if } t \text{ is odd}\\
        \binom{k}{2} + k(n-k-1) - 1, &\text{if } t \text{ is even}
    \end{cases}
\end{equation}
edges. Note that in both cases we have
$$
    d < 0.0005n^2.
$$
We show that then $G' = G \ast K_t$ contains a $K_3$-factor. For the rest of the proof let $W$ denote the set of vertices corresponding to $K_t$ in $G'$. We advise the reader that we will often switch between $G$ and $G'$.

Our strategy consists of three steps. In the first step we find a collection of vertex-disjoint triangles which cover every vertex from a set $L \subseteq V(G)$ consisting of vertices of \emph{low} degree, that is
$$
    L = \{v \in V(G) \colon \deg_G(v) < \ell \}.
$$
Let $U_1 \subseteq V(G')$ denote the set of all vertices in $G'$ which are not covered by these triangles. Then in the second step we find a collection of vertex-disjoint triangles in $G'[U_1]$ which cover all the vertices $I \subseteq V(G) \cap U_1$ of \emph{intermediate} degree, 
$$
    I = \{v \in V(G) \cap U_1 \colon \ell \le \deg_G(v) < h \}.
$$
Note that we are indeed looking at the degree of a vertex $v$ in the whole original graph and not just in $G[U_1]$. Again, let $U_2 \subseteq U_1$ denote the set of vertices which are not covered by this and the previous collection of triangles. Finally, we show that $G'[U_2]$ has minimum degree at least $2|U_2|/3$ which, by the Corr\'adi--Hajnal theorem~\cite{corradi} (stated in Section \ref{sec:trianfac}), implies that $G'[U_2]$ contains a triangle-factor. Overall, we obtain a $K_3$-factor of the whole graph $G'$.

Let us show why the first step is indeed possible. As every vertex in $L$ is missing at least $n - 1 - \ell$ edges, we must have
$$
    \frac{|L| (n - 1 - \ell)}{2} \le d\leq 0.0005n^2.
$$
This implies $|L| \le \alpha n$. Let $M_1 \subseteq L$ be a largest subset such that the induced graph $G[M_1]$ contains a perfect matching. Next, consider a largest matching in the bipartite graph induced by $L \setminus M_1$ and $V(G) \setminus L$ and let $M_2 \subseteq L \setminus M_1$ be the set of vertices contained in this matching. Finally, let $M_3 = L \setminus (M_1 \cup M_2)$ be the remaining vertices in $L$. For brevity, denote the size of $M_1, M_2$ and $M_3$ by $m_1, m_2$ and $m_3$, respectively. Note that if $m_1/2 + m_2 + 2m_3 \le t$ then there exists a collection of vertex-disjoint triangles in $G'$ which cover every vertex in $L$: For every vertex in $M_3$ we choose two vertices in $W$; for every vertex in $M_2$ it suffices to choose one additional vertex in $W$ (the third vertex in a triangle comes from the matching which saturates $M_2$); for every edge in a perfect matching from $G[M_1]$ we choose one additional vertex in $W$. 

We now show that
\begin{equation} \label{eq:m1m2m3}
    m_1/2 + m_2 + 2m_3 \le t
\end{equation}
holds. We do this by estimating the number of missing edges in terms of $m_1, m_2$ and $m_3$ and comparing it to \eqref{eq:missing}. First, observe that $L \setminus M_1$ is an independent set as otherwise we get a contradiction with the maximality of $M_1$. This gives us a simple bound of at least $\binom{m_2 + m_3}{2}$ missing edges in $G[L]$. Next we estimate the number of missing edges between $L$ and $V(G) \setminus L$. Based on the definition of $L$, we have that every vertex in $M_1$ is missing at least $n - |L| - \ell$ edges with the other endpoint being in $V(G) \setminus L$. As the size of the largest matching in the bipartite graph $B \subseteq G$ induced by $L \setminus M_1$ and $V(G) \setminus L$ is of size $m_2$, by K\H{o}nig's theorem we have that the size of a smallest vertex cover in $B$ is also of size $m_2$. Recall that the vertex cover is a set of vertices which touches every edge of the graph. Therefore, there is a set of $m_2$ vertices which touch every edge in $B$. As every vertex from $L \setminus M_1$ has degree less than $\ell$ and every vertex in $V(G) \setminus L$ has degree at most $|L \setminus M_1| \le \alpha n < \ell$ (in $B$), we conclude that $e(B) \le \ell m_2$. Therefore there are at least
$$
    (m_2 + m_3) (n - |L|) - \ell m_2
$$
edges missing between $M_2 \cup M_3$ and $V(G) \setminus L$. All together, we obtain that $G$ is missing at least
\begin{equation} \label{eq:missingL}
    \binom{m_2 + m_3}{2} + m_1(n - |L| - \ell) + (m_2 + m_3)(n - |L|) - \ell m_2
\end{equation}
edges. At this point it is convenient to parametrize the previous quantity in terms of $|L| = s$, $m_1$ and $m_2$:
$$
    f(m_1, m_2, s) = \binom{s - m_1}{2} + s(n - s) - \ell (m_1 + m_2).
$$
It is straightforward to see that $f(m_1, m_2, s)$ equals the quantity in \eqref{eq:missingL}. Note that $f$ is decreasing in $m_2$. Assume now, towards a contradiction, that $m_1/2 + m_2 + 2m_3 > t$. Then
$$
    m_2 > t - 2m_3 - m_1/2 = t - 2(s - m_1 - m_2) - m_1/2 = t - 2s + 3m_1/2 + 2m_2,
$$
hence
$$
    m_2 < 2s - t - 3m_1/2.
$$
As $m_1$ is an even integer, we can deduce that $m_2\leq 2s-t-3m_1/2-1$. Therefore, the number of missing edges in $G$ is at least
\begin{align*}
    f(m_1, m_2, s) &\ge f(m_1, 2s - t - 3m_1/2 - 1, s) \\
    &= \binom{s - m_1}{2} + s(n-s) - \ell(2s - t - m_1/2 - 1) \\
    &:= g(m_1, s).
\end{align*}
Next, observe that $g$ is increasing in $m_1$ for $0 
\le m_1 \le s$:
$$
    \frac{\partial g}{\partial m_1} = m_1  - s + (\ell+1)/2 > 0,
$$
where the second inequality follows from
$$
    s \le \alpha n \le \ell / 2.
$$
Therefore, we have
$$
    g(m_1, s) \ge g(0, s) = \binom{s}{2} + s(n-s) - \ell(2s - t - 1) =:w(s).
$$
The function $w$ is increasing in $s$ in the interval $0 \le s \le n - 2\ell - 1/2$. Note that we have $s = |L| \le \alpha n < n - 2\ell - 1/2$. Moreover, we need only consider the case $s \ge k$: If $s \le k - 1$, then $2s \le t$ thus we can trivially cover every vertex in $L$ with a triangle which uses two new vertices from $W$. Since the function $w$ is increasing, its minimum is achieved when $s=k$, i.e.
$$
    w(s) \ge \binom{k}{2} + k(n - k) - \ell(2k - t - 1).
$$
If $t$ is odd then
$$
    w(s) \ge \binom{k}{2} + k(n-k).
$$
As $w(s)$ is a lower bound on the number of missing edges in $G$, we get a contradiction with \eqref{eq:missing}. Suppose now that $t$ is even. Unlike in the previous case, we do not immediately get a contradiction with \eqref{eq:missing} as $w(k)$ is, in fact, smaller than $d$. However, $w(k+1)$ is larger than $d$, thus we only need to consider the case $|L| = k = t/2 + 1$. If $m_1 > 0$ then $m_1 \ge 2$ (because it is even), in which case we get
$$
    m_1/2 + m_2 + 2m_3 \le m_1/2 + 2(k - m_1) \le t.
$$
Therefore, we can assume $m_1 = 0$. Similarly, if $m_2 \ge 2$, then again $m_2 + 2m_3 \le t$. If $m_2 = 0$, then $L$ is a set of isolated vertices in $G$, thus the number of missing edges is at least
$$
    \binom{k}{2} + k(n-k) > \binom{k}{2} + k(n-k-1) = d,
$$
which is a contradiction. Finally, we can assume $m_2 = 1$. In particular, K\H{o}nig's theorem implies that the vertex cover number of the bipartite graph induced by $L$ and $V(G) \setminus L$ is exactly $1$, i.e.~ every edge in this bipartite graph touches some vertex $v$. If $v \in V(G)\setminus L$, then this bipartite graph is missing at least $(n-k-1)k$ edges. Together with the fact that $L$ is an independent set, we get a contradiction with $d$. Otherwise, suppose $v \in L$. If there is an edge in the neighbourhood of $v$ in $V(G) \setminus L$, then we could saturate $v$ without using any vertex from $W$, which gives enough space in $W$ to saturate $L \setminus \{v\}$. Let us denote the number of neighbours of $v$ in $V(G) \setminus L$ by $p$. The number of missing edges in $G$ is now at least
$$
    \binom{k}{2} + \binom{p}{2} + k(n-k) - p,
$$
which is easily seen to be larger than $d$ for any positive value of $p$. This exhausts all the possibilities and we can conclude that there exists a collection of vertex-disjoint triangles which saturate all vertices in $L$.

Next, let $U_1 \subseteq V(G')$ denote the set of vertices which do not belong to previously obtained triangles. Recall that every vertex $v \in I$ has degree at most $h$ in $G$, thus the number of missing edges incident to $I$ is at least
$$
    \frac{|I|(n - 1 - h)}{2} \le d\leq 0.0005n^2.
$$
This implies $|I| \le \beta n$. We saturate vertices in $I$ by sequentially choosing triangles in a greedy way by only using edges in $G$. Suppose we have saturated some $i \le |I| - 1$ vertices from $I$. Choose an arbitrary unsaturated vertex $v \in I$. As the degree of $v$ in $G$ is at least $\ell$, we have that there are still
$$
    \ell - 3|L| - 3 i \ge \ell - 3\alpha n - 3\beta n \geq \gamma n
$$
`available' neighbors of $v$ in $G$ (i.e.\ neighbors of $v$ which are not part of any triangle so far). In order to saturate $v$ we just need to find an edge within its available neighbors. However, this is possible as otherwise there are at least
$$
    \binom{\gamma n}{2} > 0.0005n^2\geq  d
$$
edges missing in $G$. To conclude, so far we have found a collection of vertex-disjoint triangles which saturate every vertex in $L \cup I$. 

Finally, let $U \subseteq V(G')$ be the set of vertices used up so far. Note that 
$$
    |L| + |I| \le |U| \le 3|L| + 3|I|\leq 3(\alpha+\beta)n.
$$
We are left with $y = n + t - |U|$ vertices and the minimum degree in the remaining graph is at least $h + t - |U| \ge 2y/3$ (where the last inequality follows from $h-\alpha n-\beta n \geq 2n/3$). By the Corr\'adi-Hajnal theorem, $G' \setminus U$ contains a triangle-factor, which finishes the proof.
\end{proof}

Let us give the constructions which show the optimality of Theorem \ref{thm:corrhajdef}. If $t$ is odd then let $G$ be a graph obtained from $K_n$ by taking an arbitrary set of $\frac{t+1}{2}$ vertices an deleting all edges incident to them. Otherwise, if $t$ is even then take a set $S \subseteq V(K_n)$ of size $t/2 + 1$ and a vertex $v \in V(K_n) \setminus S$, and delete all the edges incident to vertices in $S$ save the ones which are also incident to $v$.



\section{Concluding remarks}\label{sec:conc}
There are many interesting problems that remain open. We proved the deficiency analogue of the Corr\'adi-Hajnal theorem in the case $t$ is small. Similarly as in the case of Hamilton cycles (Theorem \ref{thm:hami}), there is another natural construction that is better for larger values of $t$: delete all edges inside a set of size $\frac{n+t}{3} + 1$. We believe that one of the three constructions is always optimal (recall that for small $t$ our construction depends on the parity of $t$). The results of Treglown~\cite{treglown} and Kierstead--Kostochka~\cite{kirkos} might be useful for obtaining further improvements on deficiency results on $K_k$-factors. A slightly different question in the same spirit is, given a graph $G$ with $m$ edges and $n$ vertices, how many vertex-disjoint triangles can we find in $G$?

There are other natural global spanning properties one might consider. For example, let $\PP$ be the property that a $k$-uniform hypergraph $\HH$ has a perfect matching, i.e.~a collection of disjoint edges covering every vertex of $\HH$. The deficiency problem then is as follows: given a $k$-uniform hypergraph $\HH$ with $n$ vertices and an integer $t$ such that $\HH\ast K_t$ does not have a perfect matching, at most how many edges can a $\HH$ have? Note that here we do not get a new question. It is easy to see that this problem is precisely equivalent to the famous Erd\H{o}s Matching Conjecture~\cite{emp} (see also Chapter 9 in~\cite{franklbook}).

In this paper we completely solve the problem of determining the deficiency of graphs with respect to Hamiltonicity. It would be interesting to obtain the solution for the analogous question in 3-uniform hypergraph setting. Here one might want to study this problem for both loose and tight Hamilton cycles. For a survey on known results on Hamiltonian cycles in hypergraphs see~\cite{kuhn2014hamilton}. In particular, Han and Zhao~\cite{hanzhao} have found new bounds on the minimum degree threshold that guarantees Hamiltonian cycles in uniform hypergraphs. Finally, another natural question to consider is the deficiency for powers of a Hamiltonian cycle.

\bibliography{mybib}
\bibliographystyle{abbrv}

\end{document}